\documentclass[11pt,reqno,tbtags]{amsart}
\usepackage{amssymb}
\usepackage[numeric,initials,nobysame]{amsrefs}
\usepackage{amsmath,amssymb,amsfonts,dsfont}
\usepackage{enumerate,graphicx}

\numberwithin{equation}{section}

\newtheorem{maintheorem}{Theorem}

\newtheorem{theorem}{Theorem}[section]
\newtheorem*{theorem*}{Theorem}
\newtheorem{lemma}[theorem]{Lemma}

\newtheorem{definition}[theorem]{Definition}

\theoremstyle{definition}{

\newtheorem*{remark*}{Remark}
}

\theoremstyle{remark}{

}

\newcommand{\E}{\mathbb{E}}
\renewcommand{\P}{\mathbb{P}}
\DeclareMathOperator{\var}{Var} \DeclareMathOperator{\Cov}{Cov} 

\renewcommand{\epsilon}{\varepsilon}

\newcommand{\cG}{{\mathcal{G}}}

\newcommand{\cF}{{\mathcal{F}}}
\newcommand{\given}{\, \big| \,}

\newcommand{\one}{\boldsymbol{1}}

\newcommand{\deq}{\stackrel{\scriptscriptstyle\triangle}{=}}
\newcommand{\K}{\mathcal{K}}
\newcommand{\GC}{{\mathcal{C}_1}} 
\newcommand{\tGC}{{\tilde{\mathcal{C}}_1}}
\newcommand{\TC}[1][\mathcal{C}_1]{#1^{(2)}} 

\newcommand{\Bin}{\operatorname{Bin}}
\newcommand{\Exp}{\operatorname{Exp}}
\newcommand{\Geom}{\operatorname{Geom}}
\newcommand{\tx}{\text{\tt{tx}}}
\newcommand{\dir}[1]{\bar{#1}}
\newcommand{\de}{{\dir{e}}}
\DeclareMathOperator{\dist}{dist}
\DeclareMathOperator{\diam}{diam}

\begin{document}

\title[Diameters in supercritical random graphs]{Diameters in supercritical random graphs \\via first passage percolation}

\author{Jian Ding, \thinspace Jeong Han Kim, \thinspace Eyal Lubetzky and Yuval Peres}

\address{Jian Ding\hfill\break
Department of Statistics\\
UC Berkeley\\
Berkeley, CA 94720, USA.}
\email{jding@stat.berkeley.edu}
\urladdr{}

\address{Jeong Han Kim\hfill\break
Department of Mathematics, Yonsei University, Seoul 120-749 Korea, and\hfill\break
National Institute for Mathematical Sciences, Daejeon 305-340, Korea.}
\email{jehkim@yonsei.ac.kr}
\urladdr{}
\thanks{Research of J.H.\ Kim was supported by a Basic Science Research Program through the National Research Foundation of Korea (NRF) funded by the Ministry of Education, Science and Technology (CRI, No. 2008-0054850).}

\address{Eyal Lubetzky\hfill\break
Microsoft Research\\
One Microsoft Way\\
Redmond, WA 98052-6399, USA.}
\email{eyal@microsoft.com}
\urladdr{}

\address{Yuval Peres\hfill\break
Microsoft Research\\
One Microsoft Way\\
Redmond, WA 98052-6399, USA.}
\email{peres@microsoft.com}
\urladdr{}

\begin{abstract}
We study the diameter of $\GC$, the largest component of the Erd\H{o}s-R\'enyi random graph $\cG(n,p)$ in the emerging supercritical phase, i.e., for $p = \frac{1+\epsilon}n$ where $\epsilon^3 n \to \infty$ and $\epsilon=o(1)$.
This parameter was extensively studied for fixed $\epsilon > 0$, yet results for $\epsilon=o(1)$ outside the critical window were only obtained very recently. Prior to this work, Riordan and Wormald gave precise estimates on the diameter, however these did not cover the entire supercritical regime (namely, when $\epsilon^3 n\to\infty$ arbitrarily slowly). {\L}uczak and Seierstad estimated its order throughout this regime, yet their upper and lower bounds differed by a factor of $\frac{1000}7$.

We show that throughout the emerging supercritical phase, i.e.\ for any $\epsilon=o(1)$ with $\epsilon^3n\to\infty$, the diameter of $\GC$ is with high probability asymptotic to $D(\epsilon,n)=(3/\epsilon)\log(\epsilon^3 n)$.
This constitutes the first proof of the asymptotics of the diameter valid throughout this phase.
The proof relies on a recent structure result for the supercritical giant component, which reduces the problem of estimating distances between its vertices to the study of passage times in first-passage percolation.
The main advantage of our method is its flexibility. It also implies that in the emerging supercritical phase the diameter of the 2-core of $\GC$ is w.h.p.\ asymptotic to $\frac23 D(\epsilon,n)$, and the maximal distance in $\GC$ between any pair of kernel vertices is w.h.p.\ asymptotic to $\frac{5}9D(\epsilon,n)$.

\end{abstract}

\maketitle

\vspace{-0.75cm}

\section{Introduction}

The Erd\H{o}s-R\'enyi random graph $\cG(n,p)$ is perhaps the most fundamental random graph model, and its rich behavior has been studied in numerous papers since its introduction in 1959 \cite{ER59}. One of the most famous phenomena exhibited by this model is the \emph{double jump} at the critical $p=1/n$. As discovered by Erd\H{o}s and R\'enyi in their celebrated papers from the 1960's, for $p=c/n$ with $c$ fixed, the largest component $\GC$ has size $O(\log n)$ with high probability (w.h.p.) when $c<1$, its size is w.h.p.\ linear in $n$ for $c>1$, and for $c=1$ its size has order $n^{2/3}$ (the latter was fully established only decades later by Bollob\'as \cite{Bollobas84} and {\L}uczak \cite{Luczak90}). Furthermore, Bollob\'as found that the critical behavior extends throughout the regime where $p=(1\pm\epsilon)/n$ for $\epsilon = O(n^{-1/3})$, known as the \emph{critical window} (or scaling window).

Despite the intensive study of this model, one of its key features --- the diameter of the largest component --- remained unknown in a regime just beyond criticality, namely for $p=(1+\epsilon)/n$ where $\epsilon=o(1)$ and $\epsilon^3 n\to\infty$ arbitrarily slowly.
Our main results determine the asymptotic behavior of the diameter throughout this regime, as well as the diameter of the 2-core and the maximal distance in it between kernel vertices (see definitions below).

\begin{maintheorem}\label{mainthm-1}
Let $\GC$ be the largest component of the random graph $\cG(n,p)$ with $p = (1 + \epsilon)/n$, where $\epsilon^3 n\to \infty$ and $\epsilon = o(1)$. Then w.h.p., \[\diam(\GC) = \frac{3+o(1)}\epsilon\log(\epsilon^3 n)\,.\]
\end{maintheorem}

A key advantage of our method over alternative approaches is its flexibility, demonstrated by the following theorem.
Recall that the \emph{2-core} of $\GC$, denoted by $\TC$, is the maximal subgraph of $\GC$ where every vertex has degree at least degree 2. The \emph{kernel} $\K$ is the multigraph obtained from $\TC$ by contracting every 2-path (a path where all interior vertices have degree 2) into an edge.

\begin{maintheorem}\label{mainthm-2}
Let $\TC$ be the $2$-core of the largest component $\GC$ of $\cG(n,p)$ with $p = (1 + \epsilon)/n$, where $\epsilon^3 n\to \infty$ and $\epsilon = o(1)$. Let $\K$ denote the kernel of $\TC$. Then w.h.p.,
\begin{align}
\diam(\TC) &= \frac{2+o(1)}\epsilon\log(\epsilon^3 n)\,,\label{eq-diam-2-core}\\
  \max_{u,v\in \K} \dist_{\GC}(u,v) &= \frac{5/3+o(1)}\epsilon\log (\epsilon^3 n)\label{eq-diam-kernel}\,.
\end{align}
\end{maintheorem}
(In the above statements, the term w.h.p.\ denotes a probability tending to $1$ as $n\to\infty$, while the $o(1)$-terms denote functions going to $0$ as $n\to\infty$.)

Theorem~\ref{mainthm-1} completes a long list of studies of the diameter in sparse Erd\H{os}-R\'enyi random graphs.
In the subcritical case, $p=(1-\epsilon)/n$ where for $\epsilon\to0$ and $\epsilon^3n\to\infty$, {\L}uczak \cite{Luczak97} obtained a precise estimate of $\log_{1/(1-\epsilon)}(2\epsilon^3 n) + O(1/\epsilon)$ for the largest diameter of a component (including the limiting distribution of the additive $O(1/\epsilon)$-term). In the critical window, it was
shown in \cite{NP} that the diameter of $\GC$ has order $n^{1/3}$. See also the recent work~\cite{ABG} studying the limiting distribution within the critical window.

However, analyzing the diameter in the supercritical case is considerably more delicate. The asymptotics of this parameter in the fully-supercritical regime, $p=(1+\epsilon)/n$ where $\epsilon > 0$ is fixed, were obtained in \cite{FeRa} (see also \cites{CL,BJR}).
Results for the regime $\epsilon=o(1)$ were only obtained fairly recently. In 2008, Riordan and Wormald~\cite{RW1} proved accurate estimates of the diameter for most of this regime, but did not cover the entire range where the random graph emerges from the critical window (i.e., $\epsilon^3 n\to\infty$ arbitrarily slowly). Note that, while the gap that remained was extremely small, the authors stated that ``our method does seem to need some concrete lower bound on $\epsilon^3 n$ tending to infinity as a function of $n$''. {\L}uczak and Seierstad \cite{LS} gave estimates for the diameter that do apply to the entire supercritical regime, yet their upper and lower bounds differ by a factor of $\frac{1000}7$.

\begin{remark*}
Following the completion of this work, Riordan and Wormald~\cite{RW2} managed to extend their analysis to the entire supercritical regime,
thus obtaining a version of Theorem~\ref{mainthm-1} with more accurate error-term estimates.
Referring to the present work, they stated in~\cite{RW2} that ``Seeing this paper stimulated us
to remove the unnecessary restriction on $\epsilon$''. We emphasize that the proofs in~\cites{RW1,RW2} rely on branching process analysis and are quite different from our methods.
\end{remark*}

\begin{remark*}
By our results and the duality of the supercritical and subcritical regimes, it follows that in the setting of Theorem~\ref{mainthm-1}, the diameter of $\GC$ is w.h.p.\ the largest diameter of a component -- larger by an asymptotic factor of $3$ compared to the largest diameter of any other component. Compare this to the subcritical case $p=(1-\epsilon)/n$, where the largest component $\GC$ has diameter of order $(1/\epsilon)\sqrt{\log(\epsilon^3 n)}$ w.h.p.\ (since $\GC$ is w.h.p.\ a tree whose size has order $\epsilon^{-2}\log(\epsilon^3 n)$, and conditional on this it is uniformly distributed on all trees of this size).
\end{remark*}

In order to establish Theorems~\ref{mainthm-1} and \ref{mainthm-2}, we apply a recent structure result proved in a companion paper \cite{DKLP}. This result translates the supercritical giant component into a contiguous tractable model constructed in 3 steps as follows (see Theorems~\ref{thm-struct} and \ref{thm-struct-gen} for the precise formulation):
 \begin{enumerate}[\indent 1.]
   \item Select a multigraph $\K$ uniformly among all graphs with a prescribed size and degree sequence, where almost all degrees are $3$.
   \item Replace the edges of $\K$ by paths of i.i.d.\ geometric lengths.
   \item Attach a Poisson-Galton-Watson tree to each vertex.
 \end{enumerate}
Note that Step~1 above constructs the kernel, Step~2 gives the 2-core and Step~3 produces the entire giant component.

Using this tool, the problem of estimating distances in $\GC$ is reduced to the study of passage times in first-passage percolation
(see, e.g., \cite{Kesten} for further information on this thoroughly studied topic). For instance, one may readily deduce from known results \cite{BHvdH} on first-passage percolation that the typical distance between kernel vertices in the 2-core is asymptotically
  $(1/\epsilon)\log(\epsilon^3n)$ (in fact, the limiting distribution of this typical distance is completely determined in~\cite{BHvdH}); see \cite{DKLP}*{Corollary 2}.

However, in weighted random graphs, maximal distances exhibit a behavior different from typical distances.
In order to prove our main results, we establish sharp large deviation estimates for these distance variables.

The rest of this paper is organized as follows.
Section~\ref{sec:perlims} contains a few preliminary facts required for the proofs. In Section~\ref{sec:rand-reg} we analyze typical and maximal distances in weighted random regular graphs, as well as in metric graphs (obtained by replacing each weighted edge by a line segment with the corresponding length). Section~\ref{sec:young-giant} contains the proofs of the main theorems in the special case of $\epsilon=o(n^{-1/4})$, where the description of the giant component has a particularly elegant form (Theorem~\ref{thm-struct}).
We extend these results to the general case of any $\epsilon=o(1)$ in Section~\ref{sec:giant}.

\section{Preliminaries and notation}\label{sec:perlims}

A random $d$-regular $G\sim\cG(n,d)$ is a graph uniformly chosen among all graphs on $n$ vertices
in which every vertex has degree $d$. One of the main tools for sampling from this distribution, as well as analyzing the behavior its typical elements, is the \emph{configuration model}, introduced by Bollob\'as \cite{Bollobas1} (see \cites{Bollobas2,JLR,Wormald99}).

To construct a graph using this method, associate each of the $n$ vertices with $d$ distinct \emph{half-edges}, and select a uniform perfect matching on these half-edges. The resulting (multi)graph is obtained by contracting each $d$ half-edges into their corresponding vertex (possibly introducing multiple edges or self-loops). Crucially, given that the graph produced is simple, it is uniform over $\cG(n,d)$, and for $d$ fixed, the probability of this event is bounded away from $0$. Hence, events that hold w.h.p.\ for the graph obtained via the configuration model also hold w.h.p.\ for $G \sim \cG(n,d)$.

One particularly useful property of the configuration model is that it allows one to construct the graph gradually, exposing the edges of the perfect matching one at a time. This way, each additional edge is uniformly distributed among all possible edges on the remaining (unmatched) half-edges.

The distance between two vertices $u,v$ in an unweighted (undirected) graph $G$, denoted by $\dist_G(u,v)$, is the number of edges in the shortest path connecting these two vertices. We will use the abbreviation $\dist(u,v)$ when there is no danger of confusion.
If $G$ has non-negative weights on its edges, the \emph{length} of a path is replaced by its weight (the sum of weights along its edges), and $\dist(u,v)$ is analogously defined as the weight of the shortest (least heavy) path between $u,v$.

The diameter of a graph $G$, denoted by $\diam(G)$, is the maximum of $\dist(u,v)$ over all possible vertices $u,v$.

It is well-known (and easy to show) that $\cG(n,d)$ is locally tree-like around a typical vertex. To formalize such statements, we use the following notion.
The \emph{tree excess} of a connected set $S$, denoted by $\tx(S)$, is the maximum number of edges that can be deleted from the induced subgraph on $S$ while still keeping it connected (i.e., the number of extra edges in that induced subgraph beyond $|S|-1$).

\section{Random regular graphs with exponential weights}\label{sec:rand-reg}

\subsection{Diameters of weighted graphs}
In this section, we consider a random regular graph for $d \geq 3$ fixed, with i.i.d.\ rate $1$ exponential variables
on its edges. In fact, here and throughout the paper, we will consider a random $d$-regular \emph{multigraph} generated via the configuration model (which will prove useful in capturing the geometry of the kernel of the largest component in $\cG(n,p)$).
Our goal is both to obtain the asymptotic diameter in this graph, and crucially, also establish its decay rate
(see Eq.~\eqref{eq-diam-decay}).

\begin{theorem}\label{thm-diameter}
Fix $d \geq 3$ and let $G \sim \cG(n,d)$ be a random $d$-regular multigraph with $n$ vertices and i.i.d.\ rate $1$ exponential variables on its edges. Then w.h.p., $\diam(G) = \big(\tfrac{1}{d-2}+\tfrac2{d}\big)\log n + O(\log\log n)$.
\end{theorem}

To prove the above result, we need to address the exponential decay of the distance between vertices, and introduce the following definition: Let $t > 0$. The
\emph{$t$-radius neighborhood} of a vertex $u$, denoted by $B_u(t)$, is
\begin{equation}
  \label{eq-t-radius-neighborhood}
  B_t(u) \deq \{ v : \dist(u,v) \leq t \}\,.
\end{equation}
Further define the threshold of $B_t(u)$ reaching a certain size as
\begin{equation}
  \label{eq-tau-u}
  T_u(s) \deq \min\{ t : |B_t(u)| \geq s \}\,.
\end{equation}
Given these definitions, we can now formulate the exponential decay of $\diam(G)$ as well as $T_u$ for all $u\in V$.
\begin{theorem}\label{thm-decay}
Fix $d \geq 3$ and let $G \sim \cG(n,d)$ be a a random $d$-regular multigraph with $n$ vertices and i.i.d.\ rate $1$ exponential variables on its edges. Then there exists some $c > 0$ so that the following holds. For any $\ell > 0$,
\begin{equation}
   \label{eq-diam-decay}
   \P\left( \diam(G) \geq \big(\tfrac{1}{d-2}+\tfrac2{d}\big)\log n + 15 \log\log n + \ell \right) \leq c \mathrm{e}^{-\ell/2} + c n^{-1/2} \,,
 \end{equation}
and for $q = 2\sqrt{d n \log n}$ and a uniformly chosen vertex $u$,
\begin{equation}\label{eq-tau-decay-uniform-u} \P\left( T_u(q) \geq \frac{\log n}{2(d-2)} + 7\log\log n + \ell \right) \leq
c \mathrm{e}^{-d \ell} + \frac{c}{n}\mathrm{e}^{-\ell} + 2n^{-3/2}\,.
\end{equation}
\end{theorem}
\begin{proof}
We first wish to prove \eqref{eq-tau-decay-uniform-u}. Fix a vertex $u$, and consider the following continuous-time exploration process. At time $t =0$, we have a ball that contains only $u$. For $t > 0$, our ball contains every $v$ that has $\dist(u,v) \leq t$, i.e., it is precisely $B_u(t)$.

In our setting, $u$ is a fixed vertex and both the graph and its weights are random, hence we can expose the edges and their weights as we grow $B_t(u)$. This leads to an equivalent description of the process:
\begin{itemize}
  \item Start with $B_0(u) = \{ u \}$,  where $u$ has $d$ (unmatched) half-edges.
  \item Reveal any matchings (and weights) of these $d$ half-edges connecting them amongst themselves (self-loops at $u$).
  \item Repeat the following \emph{exploration step}:
  \begin{itemize}
    \item Given there are $m$ half-edges in the current set, denoted by $h_1,\ldots,h_m$, let $\Psi\sim \Exp(m)$ be a rate $m$ exponential variable.
    \item Select a uniform half edge $h_i$ and match it to a uniformly chosen half-edge outside of $B$, thus introducing a new vertex with $d-1$ new half-edges to the set $B$.
    \item Reveal the matchings (and weights) of any of the half-edges of the newly added vertex whose match is also in $B$.
  \end{itemize}
\end{itemize}
To verify the validity of the above process, consider the usual configuration model for generating the neighborhood of $u$, a fixed vertex in a random $d$-regular graph. That process would repeatedly select a uniform unmatched half-edge in the current neighborhood $B$ and match it to a uniform half-edge in the entire graph. Equivalently, one can grow the neighborhood by repeatedly revealing every edge within the induced subgraph on $B$ before proceeding to expose edges to new vertices. This method of generating the underlying random regular graph will simultaneously provide us with the weights along the edges. Indeed, we can continuously grow the weights of the half-edges $h_1,\ldots,h_m$ in $B$ until one of their rate 1 exponential clocks fires. Since the minimum of $m$ exponentials is exponential with rate $m$, this is the same as choosing a uniform half-edge $h_i$ after time $\Psi$ (recall that by our conditioning, these $m$ half-edges do not pair within themselves). Note that the final weight of an edge is accumulated between the time of arrival of its first half-edge and the time of its pairing (except edges going back into $B$ whose weights are revealed immediately). Finally, the memoryless property of the exponential distribution guarantees that the process indeed generates a random $d$-regular graph with i.i.d.\ rate 1 exponentials.

Let $\tau_i$ denote the time of the $i$'th exploration step ($i \geq 0$) in the above continuous-time process, and notice
that for each $i$, at time $\tau_{i+1}$ we match a uniformly chosen half-edge from the set $B_{\tau_i}$ to a uniformly chosen half-edge among all other half-edges (excluding those in $B_{\tau_i}$).
Moreover, given $\cF_{\tau_i}$, we have that $\tau_{i+1}- \tau_i$ is an exponential variable with rate $k$, where
 $k$ is the number of half-edges in $B_{\tau_i}(u)$.

By the uniform choice of the matching, given $\cF_{\tau_i}$, the number of half-edges introduced by the new vertex at time $\tau_{i+1}$ and connecting back to $B_{\tau_{i+1}}$ is stochastically dominated by a binomial variable
\[ \Bin(d-1, \alpha)~,\mbox{ where $\alpha=\frac{d+ (d-2) (i+1)}{d n - 2i} \leq \frac{(i+2)d-2}{dn} \leq \frac{i+2}n$}\,, \]
where the first inequality above is valid for, say,  every $i \leq \frac{n}2 - 5$. To justify this, observe that after $i$ steps we have some $m \leq d + (d-2)i$ half-edges in $B_{\tau_i}$ out of some $M \geq d n - 2i$ half-edges in total. Moreover, if $b$ half-edges connected back (did not produce a new vertex) then $m = d+(d-2)i - 2b$ whereas $M = dn - 2i - 2b$. Since the probability for each of the new $d-1$ half-edges to connect back into $B_{\tau_i}$ is at most $\frac{m+d-2}{M} \leq \frac{d+(d-2)i}{dn-2i}$ the above statement holds.

Therefore, for every $i$, the tree-excess of $B_{\tau_i}$ is stochastically dominated by a binomial variable
$\Bin(i d,(i+2)/n)$ (with room to spare). Recalling that $q=2\sqrt{d n \log n}$ and defining
\begin{align*}
 r \deq \log^3 n \,,
\end{align*}
we have the following for large $n$
\begin{align}
\begin{array}{l}
\P( \tx(B_{\tau_r}) \geq 1 ) \leq \P\big(\Bin\big(d r,\frac{r+2}n\big) \geq 1\big) \leq O\big(\frac{r^2}n\big) = O\big(\frac{\log^6 n}n\big)\,,\\
\noalign{\medskip}
\P( \tx(B_{\tau_r}) \geq 2 ) \leq \P\big(\Bin\big(d r,\frac{r+2}n\big) \geq 2\big) \leq O\big(\frac{r^4}{n^2}\big) = o\big(n^{-3/2}\big)\,.
\end{array}\label{eq-tx-B-tau-r}
\end{align}
Furthermore, for any $k$ satisfying $r \leq k \leq 10q$,
\[ \P\big( \tx(B_{\tau_k}) \geq \tfrac{k}{\sqrt{r}} \big) \leq \P\big(\Bin\big(dk,\tfrac{k+2}n\big) \geq \tfrac{k}{\sqrt{r}}) \leq \exp\left(-\tfrac13 k/\sqrt{r}\right) < n^{-6}\,,\]
where the last two inequalities hold for any sufficiently large $n$ by Chernoff's inequality (see, e.g., \cite{AS}), noting that
$k^2 / n = o(k/\sqrt{r}) $.
Define the event
\begin{equation}
  \label{eq-def-R}
  R \deq \left\{ \tx(B_{\tau_k}) < k / \sqrt{r}\mbox{ for all $r \leq k \leq 10q$} \right\}\,,
\end{equation}
and note that a union bound over all $r \leq k \leq 10q$ gives that $\P( R  ) \geq 1- n^{-5}$.
(In the rare event that the exploration from $u$ exhausts itself prior to $k=10q$, the event $R$ applies only to values of $k$ for
which $\tau_k$ is defined. This occurs with probability $O(1/n^2)$ hence may be neglected.)
At this point, we have two cases concerning $\tx(B_{\tau_r})$.

\begin{list}{\labelitemi}{\leftmargin=2em}

  \item \textbf{Case 1.} The tree excess of $B_{\tau_r}$ is $0$ and the event $R$ holds so far. Denote this event by $Q_1$.

In this case, for any $i < r$, conditioning on $\tx(B_{\tau_i}) = 0$ we have that
the number of half-edges of $B_{\tau_i}$ is $d + d i - 2 i = (d-2)i + d$,
and $\tau_{i+1}-\tau_{i} \preceq Y_i$, where
\begin{align*}Y_i \sim \Exp((d-2)i+d)&\qquad(i=0,1,\ldots,r-1)\end{align*}
 and all $Y_i$'s are independent.
(Here and in what follows, $\mu\preceq\nu$ denotes stochastic domination, i.e.\ $\int f d\mu\leq \int fd\nu$ for any increasing function $f$.)
Now, each matching that contributes to the tree excess eliminates two half-edges instead of introducing $d-1$ new ones. Thus,
for each $i \geq r$, conditioning on $\tx(B_{\tau_i}) < i/\sqrt{r}$ we get that $\tau_{i+1}-\tau_i \preceq Y_i$, where
\begin{align}Y_i \sim \Exp((d-2)i+d-2 i/\sqrt{r})&\qquad(i=r,r+1,\ldots)\label{eq-Yi-geq-r}
\end{align}
Applying the Laplace transform, and noting that for $Y \sim \Exp(\rho)$
\[ \E \mathrm{e}^{\lambda Y} = \int_0^\infty \mathrm{e}^{\lambda y} \rho \mathrm{e}^{-\rho y} dy = \frac{\rho}{\rho-\lambda}\,,\]
we obtain that
\begin{align*}
  \E \left[\mathrm{e}^{\lambda (\tau_q-\tau_1)}\mid Q_1 \right]\leq
  \prod_{i=1}^{r-1} \frac{(d-2)i+d}{(d-2)i+d - \lambda}
    \prod_{i=r}^{q-1} \frac{(d-2)i+d-2 i/\sqrt{r}}{(d-2)i+d-2 i/\sqrt{r} - \lambda}\,.
\end{align*}
Taking $\lambda = d$, we get
\begin{align*}
  \E \left[\mathrm{e}^{d (\tau_q-\tau_1)} \mid Q_1 \right]&\leq
  \prod_{i=1}^{r-1} \left(1 + \frac{d}{(d-2)i}\right)
    \prod_{i=r}^{q-1} \left(1+ \frac{d}{(d-2)i-2 i/\sqrt{r} }\right) \\
&\leq \exp \left[ \frac{d}{d-2}\sum_{i=1}^{r-1} \frac{1}i + \frac{d}{d-2}\cdot\frac{1}{1-\frac{2}{d-2}r^{-1/2}} \sum_{i=r}^{q-1}\frac1{i}\right]\\
&\leq \exp\left[\frac{d}{d-2}\left(1+ O(r^{-1/2})\right)\log q + 2\right]\leq 10 q^{d/(d-2)}~,
\end{align*}
 where the last inequality holds for large $n$. Hence, in Case~1 we have
\[\P\left( \tau_q-\tau_1 \geq \frac{\log n}{2(d-2)} + \frac{\log\log n}{2(d-2)} + \ell ~\Big|~  Q_1\right)
\leq O(\exp(-d \ell))\,.\]
As $\tau_1 \sim \Exp(d)$, altogether in this case
\[\P\left( \tau_q \geq \frac{\log n}{2(d-2)} + \frac{\log\log n}{2(d-2)} + \ell ~\Big|~ Q_1\right)
\leq O(\exp(-d \ell))\,.\]

\item \textbf{Case 2.} The tree excess of $B_{\tau_r}$ is $1$ and the event $R$ holds so far. Denote this event by $Q_2$.

Here, for any $i < r$, the following holds: Conditioned on $\tx(B_{\tau_i})\leq 1$,
the number of half-edges in $B_{\tau_i}$ is $d$ for $i=0$ and at least $(d-2)i$ for $i\geq 1$. Hence,
 $\tau_{i+1}-\tau_{i} \preceq Y_i$, where the $Y_i$'s are independent variables given by
 \begin{align*}
Y_i \sim \Exp\big((d-2)(i+1)\big)\quad(i=0,1,\ldots,r-1)\,.
 \end{align*}
For $i\geq r$, conditioned on $\tx(B_{\tau_i}) < i/\sqrt{r}$ the bounds of Case~1 hold, that is, $\tau_{i+1}-\tau_i \preceq Y_i$ for the variables $Y_i$ as given in \eqref{eq-Yi-geq-r}.
This yields that
\begin{align*}
  \E \left[\mathrm{e}^{\lambda (\tau_q-\tau_1)} \mid Q_2\right]\leq
  \prod_{i=1}^{r-1} \frac{(d-2)(i+1)}{(d-2)(i+1) - \lambda}
    \prod_{i=r}^{q-1} \frac{(d-2)(i+1)-2 i/\sqrt{r}}{(d-2)(i+1)-2 i/\sqrt{r} - \lambda}\,.
\end{align*}
Taking $\lambda = 1$, we get
\begin{align*}
  \E \left[\mathrm{e}^{ \tau_q-\tau_1} \mid Q_2\right] &\leq
  \prod_{i=1}^{r-1} \left(1 + \frac{1}{(d-2)(i+1)}\right)
    \prod_{i=r}^{q-1} \left(1+ \frac{1}{(d-2)(i+1)-2 i/\sqrt{r} -1}\right) \\
&\leq \exp \left[ \frac{1}{d-2}\sum_{i=2}^{r} \frac{1}i + \frac{1}{d-2}\cdot\frac{1}{1-2[(d-2)\sqrt{r}]^{-1}} \sum_{i=r}^{q-1}\frac1{i}\right]\\
&\leq \exp\left[\frac{1}{d-2}\left(1+ O(r^{-1/2})\right)\log q + 2\right]\leq 10 q^{1/(d-2)}~,
\end{align*}
where the last inequality holds for large $n$. Hence, in Case~2~,
\[\P\left( \tau_q-\tau_1 \geq \frac{\log n}{2(d-2)} + \frac{\log\log n}{2(d-2)} + \ell ~\Big|~ Q_2\right)
\leq O(\exp(-\ell))\,,\]
and again, as $\tau_1 \preceq \Exp(d-2)$, in this case
\[\P\left( \tau_q \geq \frac{\log n}{2(d-2)}  + \frac{\log\log n}{2(d-2)} + \ell ~\Big|~ Q_2\right)
\leq O(\exp(-\ell))\,.\]
\end{list}

Combining the above two cases using \eqref{eq-tx-B-tau-r} and that $\P(R) \geq 1-n^{-5}$, we conclude
that for some $c >0$ fixed,
\[\P\left( \tau_q \geq \frac{\log n}{2(d-2)} + 7\log\log n + \ell \right) \leq
c \mathrm{e}^{-d \ell} + \frac{c}{n}\mathrm{e}^{-\ell} + 2n^{-3/2}\,,
\]
(here the $7\log\log n$ term eliminated the $\log^6 n$ factor from Eq.~\eqref{eq-tx-B-tau-r}), readily
implying \eqref{eq-tau-decay-uniform-u}.

To obtain \eqref{eq-diam-decay} we need the next simple lemma.
\begin{lemma}\label{lem-dist-Tu-Tv}
With high probability, $\dist(u,v) \leq T_u(q) + T_v(q)$ for all $u,v$.
\end{lemma}
\begin{proof}
Assume that $T_u(q),T_v(q)<\infty$ (i.e.\ each of the connected components of $u$ and $v$ consists of at least $q$ vertices)
otherwise the statement of the lemma holds trivially.

Fix two vertices $u,v$ and consider the aforementioned exploration process.
Explore $B_t(u)$ until reaching $t = T_u(q)$, and condition on the
event that its tree-excess is $o(q)$, as ensured by the event $R$ defined in \eqref{eq-def-R}
(this event holds with probability at least $1-n^{-5}$). Thus, there are $(d-1-o(1))q$ half-edges in $B_{T_u(q)}$
except with probability $n^{-5}$.

Next, begin exposing $B_t(v)$; each matching adds a uniform half-edge to the neighborhood of $v$,
and so the probability that $B_{T_v(q)}$ does not intersect $B_{T_u(q)}$ is at most
\[ \left(1 - \frac{(d-1-o(1))q}{dn}\right)^{q} \leq \exp(- 4(d-1-o(1))\log n ) < n^{-7}\]
for any large $n$. A union bound over $u,v$ now completes the proof.
\end{proof}
To infer \eqref{eq-diam-decay} from \eqref{eq-tau-decay-uniform-u} and the above lemma, argue as follows.
Choose
$\ell = \frac1{d}\log n + k$ in \eqref{eq-tau-decay-uniform-u} to obtain that for some constant $c> 0$, a uniformly chosen vertex $u$ has
\[ \P\left( T_u(q) \geq \big(\tfrac{1}{2(d-2)}+\tfrac{1}{d}\big)\log n + 7 \log\log n + k \right) <
\frac{c}n \mathrm{e}^{-d k} + \frac{c}{n} \mathrm{e}^{-k} + 2n^{-3/2}\
\]
(with room to spare), and by taking a union bound over $u$ it follows that, for some other $c>0$, the probability of
\[ \left\{T_u(q) \leq \left(\tfrac{1}{2(d-2)}+\tfrac{1}{d}\right)\log n + 7 \log\log n + k\right\} \mbox{ for every $u$}\]
is at least $1 - c \mathrm{e}^{-k} - 2/\sqrt{n}$. Combining Lemma~\ref{lem-dist-Tu-Tv} with a choice of $k=\ell/2$ now gives the estimate~\eqref{eq-diam-decay} for the decay of the diameter, as required.
\end{proof}

We also need the next simple lemma, which bounds the number of edges in a path achieving $\dist(u,v)$ for any $u,v$.
\begin{lemma}\label{lem-dist-path-len}
With high probability, for any two vertices $u,v$, the number of edges in the path achieving $\dist(u,v)$ is at most $4d\mathrm{e} \log n$.
\end{lemma}
\begin{proof}
First consider an arbitrary given $d$-regular graph $H$ on $n$ vertices with i.i.d.\ rate 1 exponentials on its edges.
Consider a simple path consisting of $\ell$ edges, and let $X_i$ denote the $\Exp(1)$ variable that corresponds to the $i$'th edge.
Setting $S_\ell = \sum_{i\leq \ell} X_i$,
\begin{align*}
\P(S_\ell \leq a \ell ) = \P(\mathrm{e}^{-\lambda S_\ell} \geq \mathrm{e}^{-\lambda a \ell })
\leq \mathrm{e}^{\lambda a \ell}\E \mathrm{e}^{-\lambda S_\ell} = \frac{\mathrm{e}^{\lambda a \ell}}{(1+\lambda)^\ell}~,
\end{align*}
where the last equality is by the fact that $\int_0^\infty \mathrm{e}^{-\lambda x} \mathrm{e}^{-x} dx = \frac{1}{1+\lambda}$.
Choosing $\lambda = (1-a)/a$ and subsequently $a = 1/(2d\mathrm{e})$, we get that
\[ \P(S_\ell \leq a \ell ) \leq \left(a \mathrm{e}^{1-a}\right)^\ell \leq (2d)^{-\ell}\,.\]
Fix a starting position $u$ for the path. Summing over all $d (d-1)^{\ell-1}$ possible simple paths of length $\ell$
originating from $u$, we have that the probability that one of these paths would have weight smaller than $a \ell$ is
at most $2^{-\ell}$. In particular, if $A_u$ denotes the event that for some integer $\ell \geq 2\log_2 n$ there is a simple path originating from $u$ with length $\ell$ and weight at most $a \ell$, then $\P(A_u) \leq 2n^{-2}$. Taking a union bound over the vertices $u$ gives this for any vertex of the graph $H$ w.h.p.

Now take $H$ to be our random $d$-regular graph. By the estimate of Theorem~\ref{thm-decay} on the diameter we deduce that, w.h.p.,
for any $u,v$ either the path achieving $\dist(u,v)$ has at most $2 \log_2 n$ edges, or the number of edges is contains is at most
\[ \diam(G) / a \leq 2d\mathrm{e}\Big(\frac{1}{d-2}+\frac 2{d}+o(1)\Big)\log n < 4d \mathrm{e}\log n
\] for large $n$, thus concluding the proof.
\end{proof}

The upper bound on the diameter of the weighted graph $G$ will follow immediately from the results we have established so far.
For the lower bound, we need to the following lemma.

\begin{lemma}\label{lem-typical-dist-lower}
  Let $u,v$ be two uniformly chosen vertices of the graph $G$ defined in Theorem~\ref{thm-decay}.
  Let $B'_t(x) = \{ y : \dist(y, N_G(x)) \leq t\}$, where $N_G(x)$ denotes the neighbors of $x$ in $G$.
  Then w.h.p., $B'_{t_0}(u) \cap B'_{t_0}(v) = \emptyset$ for $t_0 = \frac{1}{2(d-2)}\log n - 3\log\log n$.
\end{lemma}
\begin{proof}
Fix a vertex $u$, consider the exploration process defined in the proof of Theorem~\ref{thm-decay} from the set $N_G(u)$ (w.h.p.\ comprised of $d(d-1)$ half-edges), and again let $\tau_i$ be the time of the $i$'th step.
As argued before, $\tau_{i+1}-\tau_i \succeq Y_i$, where the $Y_i$'s are independent exponential variables given by \[Y_i \sim \Exp\left((d-2)i+d(d-1)\right)\] (this follows from the fact that the worst case is when the explored set forms a tree).
At this point, setting $b \deq d(d-1)-(d-2)$, we get 
\begin{align*}
\P(\tau_z \leq t)& \leq \int_{\sum_{i=1}^z x_i \leq t} \prod_{i=1}^z \big[(d-2)i+b\big] \mathrm{e}^{-\sum_{i=1}^z ((d-2)i+b)x_i} dx_1 \ldots dx_z\\ & =  \int_{0 \leq y_1 \leq \ldots \leq y_z \leq t} \prod_{i=1}^z
\big[(d-2)i+b\big] \mathrm{e}^{-b y_z} \mathrm{e}^{-(d-2)\sum_{i=1}^z y_i} dy_1 \ldots dy_z\,,
\end{align*}
where $y_k = \sum_{i=0}^{k-1} x_{z-i}$.
Letting $y$ play the role of $y_z$ and accounting for all permutations over $y_1,\ldots,y_{z-1}$ (giving each such variable the range $[0,y]$),
\begin{align*}
\P(\tau_z \leq t) &\leq \int_{0}^t \mathrm{e}^{-(d-2+b) y}\frac{\prod_{i=1}^z (i+\frac{b}{d-2}) }{(z-1)!} \\
&\qquad \cdot \left( \int_{[0,y]^{z-1}}(d-2)^z \mathrm{e}^{-(d-2)\sum_{i=1}^{z-1} y_i} dy_1 \ldots dy_{z-1}\right)dy\\
 &\leq \int_{0}^t \mathrm{e}^{-(d-2+b) y}\frac{\prod_{i=1}^z (i+\frac{b}{d-2}) }{(z-1)!}  \cdot \bigg( \prod_{i=1}^{z-1} \int_0^y (d-2) \mathrm{e}^{-(d-2)y_i} d y_i \bigg)\\
&\leq c (d-2) z^{\frac{b}{d-2}+1} \int_0^t \mathrm{e}^{-d(d-1)y}(1 - \mathrm{e}^{-(d-2)y})^{z-1} dy\,, \end{align*} where $c > 0$ is an absolute constant.
Setting $t_0 = \frac{1}{d-2}(\log z
- 2 \log \log n)$ and  $z = \sqrt{n/\log n}$ we obtain that
\[\P(\tau_z \leq t_0) \leq c (d-2) z^{\frac{b}{d-2}+1} \int_0^{t_0}
\mathrm{e}^{-\log^2 n} d y  = o(n^{-5})\,,\]
where we used the fact that $(1-\mathrm{e}^{-(d-2)y})^{z-1} \leq \mathrm{e}^{-\log^2 n}$ for all $0\leq y \leq t_0$.

To conclude the proof, observe that the above argument showed that w.h.p.\ $|B'_{t_0}(u)| \leq z$. Choosing another uniform vertex $v$ (which misses this mentioned set with probability $1-o(1)$) and exposing $B'_{t_0}(v)$, again w.h.p.\ we obtain a set of size at most $z$. Crucially, each matching is uniform among the remaining half-edges, and so its probability of hitting the boundary of $|B'_{t_0}(u)|$ is at most $z/n$. Altogether,
\[ \P( B'_{t_0}(u) \cap B'_{t_0}(v) \neq \emptyset) \leq \frac{z^2}{n} + o(1) = o(1)\,,\]
as required.
\end{proof}

We are now ready to prove the asymptotic behavior of the diameter.
\begin{proof}[\emph{\textbf{Proof of Theorem~\ref{thm-diameter}}}]
For the upper bound, apply Theorem~\ref{thm-decay} with, say, $\ell = \log\log n$, to obtain that
w.h.p.\
\[ \diam(G) \leq \Big(\frac{1}{d-2}+\frac2{d}\Big)\log n + 16 \log\log n = \Big(\frac{1}{d-2}+\frac2{d}+o(1)\Big)\log n\,.\]
It remains to provide a matching asymptotic lower bound.

Set \[D \deq \frac{1}{d}\log n - \frac1{d}\log\log n\,,\] and call a vertex ``good'' if the weight on all the $d$ edges connected to it is larger than $D$.
For a vertex $u \in V$, let $A_u$ denote the event that $u$ is good.
Clearly, $\P(A_u) = \frac{\log n}n$, and so if $Y = \sum_u \one_{A_u}$ counts the number of good vertices, we have that $\E Y = \log n$. Furthermore,
\begin{align*}
 \var(Y) &= \sum_{u,v}\Cov(\one_{A_u}, \one_{A_v}) = \sum_u \var(\one_{A_u}) + \sum_u \sum_{ v: uv \in E(G)}\Cov(\one_{A_u},\one_{A_v} )\\
 &\leq \E Y + \sum_u d \P(A_u) = (d+1)\E Y \,,
\end{align*}
and by Chebyshev's inequality we deduce that, for instance, $Y \geq \frac23 \log n$ w.h.p. In particular, the number of pairs of distinct good vertices, denoted by $R$, satisfies $R \geq \frac14 \log^2 n$ w.h.p.

On the other hand, recalling Lemma~\ref{lem-typical-dist-lower} and taking \[t_0 = \frac{1}{2(d-2)}\log n -3\log\log n\,,\]
we have that two uniform vertices $u,v$ satisfy
$B'_{t_0}(u) \cap B'_{t_0}(v) = \emptyset$ w.h.p., where the ball $B'_t(x)$ includes all vertices of distance $t$ from the neighbors of $x$ (excluding $x$ itself). In particular, condition on the events $A_u$ and $A_v$, the probability that
$B'_{t_0}(u)$ does not intersect $B'_{t_0}(v)$ remains the same (since the weights on the immediate edges incident to $u,v$ do not play a part).
Therefore, for two uniformly chosen vertices $u,v$ we have
\[ \P\left( A_u\,,\,A_v\,,\,B'_{t_0}(u)\cap B'_{t_0}(v)\neq\emptyset\right) = o(\P(A_u\,,\,A_v))\,.\]
Denoting by $R'$ the the number of pairs of good vertices that are of distance at most $2D+\frac{1}{d-2}\log n -6\log\log n$, we deduce that
\[ \E R' = o(\E Y^2) = o((\E Y)^2) = o(\log^2 n)\,.\]
By Markov's inequality, $R' \leq \frac18 \log^2 n$ w.h.p., and hence $R - R'$ is w.h.p.\ nonempty.
This implies the existence of two vertices whose distance is at least $\big(\frac1{d-2}+\frac2{d}\big)\log n - 7\log\log n$, completing the proof.
\end{proof}


\subsection{Diameters of metric graphs}
We consider the following continuous analogue of the diameter of a weighted graph
(see, e.g., \cite{BC} for related information).

\begin{definition}
Let $G=(V,E)$ be a graph with non-negative weights on its edges $\{w(e):e \in E\}$.
The corresponding metric graph $\mathcal{X}=\mathcal{X}(G)$ is the graph obtained by replacing every $e\in E$ by a line segment $L_e$ of length $w(e)$, with the (uncountable) vertex set $\cup_e L_e$ and the obvious shortest path metric.
\end{definition}
That is, the distance between any $x,y\in\mathcal{X}$, lying on two distinct $L_e$ and $L_{e'}$ resp., is the minimum of
 $|x'-x|_{L_e}+|y'-y|_{L_{e'}}+\dist_G(x',y')$ over all 4 possible choices of endpoints $x',y'$ of $L_e,L_{e'}$ resp., where $|\cdot|_{L_e}$ is the Euclidean distance in the interval $[0,w(e)]$ and we identified $x',y'$ with vertices of $G$ (When $e=e'$ the distance is the minimum of the above and $|x-y|_{L_e}$).

For a fixed integer $d \geq 3$, Let $G \sim \cG(n,d)$ be a random $d$-regular graph
with i.i.d.\ rate $1$ exponential weights on its edges.
At times we will identify the metric graph $\mathcal{X}$ with its points (the union of its line segments).
As before, we let $\diam(\mathcal{X}) =\max_{x,y\in \mathcal{X}} \dist(x,y)$.

The next theorem establishes the typical diameter of $\mathcal{X}$, which differs from that of $G$ by a term of $(\frac{d-2}d + o(1))\log n$.
\begin{theorem}\label{thm-diameter-cont}
Fix $d \geq 3$, let $G \sim \cG(n,d)$ be a random $d$-regular multigraph on $n$ vertices with i.i.d.\ $\Exp(1)$ variables on its edges, and let $\mathcal{X}(G)$ be its metric graph. Then $\diam(\mathcal{X}) = \big(1+\tfrac{1}{d-2}\big)\log n+O(\log\log n)$ w.h.p.
\end{theorem}

The proof follows the same arguments used to prove Theorem~\ref{thm-diameter}, but instead of exposing a neighborhood of an initial vertex, it does so while excluding one of the edges incident to this vertex. We will therefore focus on the modifications required to adapt the original argument to our new setting.

Recall that we defined $T_u(q)$ in the discrete setting as the minimal $t$ such that $|B_t(u)| \geq q$.
In the continuous case, we will still have $B_t(u)$ count the number of vertices of $G$, that is, endpoints of line segments in $\mathcal{X}$.
Further define $B_t(\de)$ for a directed edge $\de = (x,y)$ as the set of all vertices of $G$ whose distance from $y$ is at most $t$, without using the edge $\de$ (in any direction). That is, $B_t(\de)$ is the set $B_t(y)$ minus vertices that depend on $\de$ for being included in this set.

A small modification is required in the rare case of a self-loop or multiple-edge incident either to $y$ or to one of its neighbors excluding $x$. In this case we call the edge $\de=(x,y)$ ``rare'' and take $B_t(\de)$ to be $B_t(\de')$, where $\de' = (y,x)$ is the reversed edge. Observe that w.h.p.\ at most one of these edges is rare as the probability for each of these independent events is $O(1/n)$.

Similar to before (yet slightly modified in the case of a rare edge), let
\begin{equation}
  \label{eq-t(q)-def}
  T_\de(q) = \min\{ t : |B_t(\de)|\geq q \} + \one_{\{\de\text{ is rare}\}} \frac12 w(\de)\,.
\end{equation}
We can now formulate an exponential decay statement for the $\diam(\mathcal{X})$, analogous to Theorem~\ref{thm-decay}.

\begin{theorem}\label{thm-decay-cont}
Let $G$ and $\mathcal{X}$ be as in Theorem~\ref{thm-diameter-cont}. Then there exists some $c > 0$ so that the following holds. For any $\ell > 0$,
\begin{equation}
   \label{eq-diam-decay-cont}
   \P\left( \diam(\mathcal{X}) \geq \big(1+\tfrac{1}{d-2}\big)\log n + 15 \log\log n + \ell \right) \leq c \mathrm{e}^{-\ell} + (2+c \mathrm{e}^{-\ell/2}) n^{-1/2} \,,
 \end{equation}
and for $q = 2\sqrt{d n \log n}$ and a uniformly chosen directed edge $\de$,
\begin{equation}\label{eq-tau-decay-uniform-e-cont} \P\left( T_\de(q) + \tfrac12 w(\de) \geq \frac{\log n}{2(d-2)} + 7\log\log n + \ell \right) \leq
c \mathrm{e}^{-2 \ell} + \frac{c}n \mathrm{e}^{-\ell} + 2n^{-3/2}\,.
\end{equation}
\end{theorem}
\begin{proof}
Consider the exploration process from the endpoint $y$ of the uniformly chosen edge $\de=(x,y)$, while disregarding $\de$. This process is identical to the process defined in the previous subsection except that it starts with $d-1$ half-edges rather than $d$ half-edges.

Letting $\tau_i$ be the time of the $i$'th exploration step, we have that there are at most $d-1+(d-2)i$ half-edges in $B_{\tau_i}$ out of $dn-2i$ half-edges in total. As before, the number of half-edges, introduced at time $\tau_{i+1}$ and connecting back to $B_{\tau_{i+1}}$, is stochastically dominated by a binomial variable
\[\Bin(d-1, \alpha)\mbox{ where $\alpha=\frac{d-1+(d-2)(i+1)}{dn-2i} \leq \frac{i +2}{n}$}\,.\]
Thus, we apply the same analysis of the tree excess of $B_{\tau_i}$ as in the proof of Theorem~\ref{thm-decay} (stochastically dominating it via a binomial $\Bin(d,(i+2)/n)$) to obtain the exact bounds of \eqref{eq-tx-B-tau-r}, as well as that $\P(R) \geq 1 - n^{-5}$ for $R$ defined in \eqref{eq-def-R}.

Examining the two cases for $\tx(B_{\tau_r})$, we now have that $\tau_{i+1}-\tau_i \preceq Y_i$ for i.i.d.\ exponential variables $Y_i$ defined as follows:
\begin{list}{\labelitemi}{\leftmargin=2em}

  \item \textbf{Case 1.} (The tree excess of $B_{\tau_r}$ is $0$ and the event $R$ holds)
   \begin{align*}
Y_0 &\sim \Exp(d-1)~,\quad Y_i \sim \Exp(d-1 + (d-2)i)   &(i=1,\ldots,r-1)\,,\\
Y_i &\sim \Exp(d-1+(d-2)i-2 i/\sqrt{r})&(i=r,r+1,\ldots)\,.
 \end{align*}
For $d\geq 3$ the rates of the above $Y_i$'s are all at least $2$.

  \item \textbf{Case 2.} (The tree excess of $B_{\tau_r}$ is $1$ and the event $R$ holds)

First assume $\de$ is not rare. In this case there are $d-1$ half-edges initially,
  the first step adds $d-2$ half-edges to the boundary ($d-1$ new ones at the cost of $1$ half-edge lost, as there are no loops or multiple edges), and similarly the second steps adds $d-2$ half-edges. Thus,
\begin{align*}
&Y_0 \sim \Exp(d-1)~,\quad Y_1 \sim \Exp(2d-3)\,,\\
&Y_i \sim \Exp(3d-7 + (d-2)(i-2))~\quad\qquad\qquad  (i=2,\ldots,r-1)\,,\\
&Y_i \sim \Exp(3d-7+(d-2)(i-2)-2 i/\sqrt{r})~\quad(i=r,r+1,\ldots)\,.
 \end{align*}
 Note that for $d\geq 3$ the rates of the above $Y_i$'s are again all at least $2$.
 In the special case where $\de$ is rare (occurring with probability $O(1/n)$),
 the analogous sequence of $Y_i$'s will feature a reduced rate of at least $1$.
\end{list}
It is then easy to verify that a choice of $\lambda = 2$ for the Laplace transform for a typical edge $\de$ and $\lambda=1$ when $\de$ is rare, plugged in the same calculations as in the proof of Theorem~\ref{thm-decay}, gives that
\begin{equation*} \P\left( T_\de(q) \geq \frac{\log n}{2(d-2)} + 7\log\log n + \ell \right) \leq
c \mathrm{e}^{-2 \ell} + \frac{c}n \mathrm{e}^{-\ell} + 2n^{-3/2}
\end{equation*}
(the middle term corresponded to a rare edge $\de$; otherwise, Cases~1,2 both give a decay-rate of $2$).

Consider an edge $\de$ that is not rate.
Crucially, $w(\de)$ is independent of $T_{\de}(q)$ by definition, and as $w(\de) \sim \Exp(1)$ (and so $\frac12w(\de)\sim \Exp(2)$ has decay rate 2) we immediately obtain \eqref{eq-tau-decay-uniform-e-cont}.
Alternatively, when $\de$ is rare, our modified definition of $T_{\de}(q)$ again verifies \eqref{eq-tau-decay-uniform-e-cont} (this time both $w(\de)$ and the growth of $B$ have rate $1$).

Choosing $\ell = \frac12 \log n + k$ and taking a union bound over $\de$, we have that
\begin{equation}\label{eq-tau(de)-uniform} \left\{ T_\de(q) + \tfrac12 w(\de) \leq \big(\tfrac1{2(d-2)} + \tfrac12\big)\log n + 7\log\log n + k \right\} \mbox{ for every $\de$}
\end{equation}
with probability at least $1 - c \mathrm{e}^{-2 k} - (2+c\mathrm{e}^{-k}) n^{-1/2}$.

To conclude \eqref{eq-diam-decay-cont} from \eqref{eq-tau(de)-uniform}, argue as follows. For two points $x,y\in\mathcal{X}$, write $u_x,u_y$ for their closest points in $G$ respectively, and let $\de_x,\de_y$ denote the directed edges containing $x,y$ and ending at $u_x,u_y$ respectively.
For simplicity, assume first that neither of these edges is rare. Clearly,
\[ \dist(x,y) \leq \dist(u_x,u_y) + \frac12 w(\de_x) + \frac12w(\de_y) \,.\]
Now, by Lemma~\ref{lem-dist-Tu-Tv}, w.h.p.\ we have $\dist(u,v) \leq T_u(q) + T_v(q)$ for all $u,v$. Furthermore, for any directed edge $\de = (u,v)$, as $B_t(\de) \subset B_t(v)$ we get that $T_\de(q) \geq T_{v}(q)$. Altogether, w.h.p.\
\begin{equation}
  \label{eq-metric-dist(x,y)}
  \dist(x,y) \leq T_{\de_x}(q) + \frac12 w(\de_x) + T_{\de_y}(q) + \frac12w(\de_y)\,.
\end{equation}
We claim that the above inequality remains valid if one or more of the edges $\de_x,\de_y$ is rare. For instance,
if $\de_x$ is rare we could replace $u_x$ by the other endpoint of $\de_x$ at a cost of at most $\frac12 w(\de_x)$, which is precisely accounted for in the modified definition of $T_{\de_x}(q)$.
The exponential decay of the diameter, as stated in \eqref{eq-diam-decay-cont}, now immediately follows from \eqref{eq-tau(de)-uniform}.
\end{proof}

\begin{proof}[\emph{\textbf{Proof of Theorem~\ref{thm-diameter-cont}}}]
The upper bound on $\diam(\mathcal{X})$ follows directly from the exponential decay estimate in \eqref{eq-diam-decay-cont}, e.g., by
taking $\ell = \log\log n$.

The proof of the lower bound will follow from essentially the same arguments used to prove Theorem~\ref{thm-diameter}. Here,
instead of defining a ``good'' vertex, we call an edge ``good'' if its weight is larger than
$ D =\log n - \log\log n$,
and obtain that the probability that a given edge $e$ is good is $\frac{\log n}{n}$. The same second moment argument that showed that $R$, the number of pairs of distinct good vertices, is at least $\frac14 \log^2 n$ w.h.p., now yields the same estimate on the number of pairs of distinct good edges.

The argument in Theorem~\ref{thm-diameter} then proceeded with an application of Lemma~\ref{lem-typical-dist-lower}, which bounds the probability that two uniformly chosen vertices $u,v$ have intersecting neighborhoods $B'_{t_0}(u)$ and $B'_{t_0}(v)$ for a prescribed $t_0$ (where $B'_t(x)$ consists of vertices that are of distance at most $t$ from $N_G(x)$, the neighbors of $x$ in $G$). We now note that the statement of this lemma holds also for any uniformly chosen points $x,y\in\mathcal{X}$, when $N_G(x)$ still stands for the neighbors of $x$ in the graph $G$ (i.e., the endpoints of the incident edges). To see this, simply replace the $b$ by $2(d-1)-(d-2)$ in the proof of that lemma, accounting for the $2(d-1)$ initial half-edges in the exploration process. Since the value of $b$ does not play a role in the proof (as long as it is a constant), the proof holds without any further changes (for the same value of $t_0$).

Therefore, as before we may deduce that the expected number of pairs of good edges whose distance is at most $2t_0$ is $o(\log^2 n)$, and so w.h.p.\ (using Markov's inequality) there exists a pair of good edges, $e,f$, whose distance in $\mathcal{X}$ is at least $2t_0$.
The proof is now concluded by choosing the middle points in $e,f$ to obtain a distance of
\[ 2\cdot \frac{D}2 + 2 t_0 = \Big( 1 + \frac{1}{d-2}\Big)\log n -7\log\log n\,,\]
as required.
\end{proof}

\section{The young giant component}\label{sec:young-giant}

In this section, we focus on the giant component in the regime $\epsilon^3 n \to \infty$ and $\epsilon = o(n^{-1/4})$.
In order to relate the results of the previous section (concerning diameters of weighted random regular graphs) to this setting, we apply the main result of the companion paper \cite{DKLP}, which provides a complete and tractable description of $\GC$.
\begin{theorem}[\cite{DKLP}*{Theorem 1}]\label{thm-struct}
Let $\GC$ be the largest component of the random graph $\cG(n,p)$ for $p = \frac{1 + \epsilon}{n}$, where $\epsilon^3 n\to \infty$ and $\epsilon = o(n^{-1/4})$. Then $\GC$ is one-sided contiguous to the model $\tGC$, constructed in 3 steps as follows:
\begin{enumerate}[1.\!]
  \item\label{item-struct-base} Let $Z \sim \mathcal{N}\left(\tfrac23\epsilon^3 n, \epsilon^3 n\right)$, and 
      select a random 3-regular multigraph $\K$ on $N = 2\lfloor Z \rfloor$ vertices.
  \item\label{item-struct-edges} Replace each edge of $\K$ by a path, where the path lengths are i.i.d.\ $\Geom(\epsilon)$.
  \item\label{item-struct-bushes} Attach an independent $\mathrm{Poisson}(1-\epsilon)$-Galton-Watson tree to each vertex.
\end{enumerate}
That is, $\P(\tGC \in \mathcal{A}) \to 0$ implies $\P(\GC \in \mathcal{A}) \to 0$
for any set of graphs $\mathcal{A}$.
\end{theorem}
In the above, a $\mathrm{Poisson}(\mu)$-Galton-Watson tree (or a PGW-tree for short)
is the family tree of a Galton-Watson branching process with offspring distribution $\mathrm{Poisson}(\mu)$, and $\mathcal{N}(\mu,\sigma^2)$ stands for the Normal distribution with mean $\mu$ and variance $\sigma^2$.

\subsection{The 2-core and its kernel} In light of the above theorem, the kernel $\K$ of the giant component in the regime of $p=\frac{1+\epsilon}n$ for $\epsilon^3 n\to\infty$ and $\epsilon=o(n^{-1/4})$ can be regarded a random 3-regular graph $G \sim \cG(N,3)$. The 2-core $\TC$ is then obtained by replacing the edges of this graph by 2-paths (i.e., paths whose interior vertices all have degree 2) of lengths i.i.d.\ geometric with mean $1/\epsilon$.

\begin{proof}
  [\emph{\textbf{Proof of Theorem~\ref{mainthm-2} for the regime $\epsilon=o(n^{-1/4})$}}]
Let $G \sim \cG(N , 3)$ be a random 3-regular graph with i.i.d.\ $\Exp(\lambda)$ edge weights $\{w(e):w\in E(G)\}$, where
$\mathrm{e}^{-\lambda} = 1-\epsilon$. Further let $H$ be the unweighted graph obtained by taking the underlying graph of $G$, and replacing each of its edges by paths of length i.i.d.\ geometric variables with mean $1/\epsilon$. Let $K$ denote the subset of the vertices of $H$ of degree at least $3$.

By Theorem~\ref{thm-struct}, for $N$ as defined as in Step~\ref{item-struct-base} (which in particular satisfies $N = (\tfrac43+o(1))\epsilon^3 n$ w.h.p.), $H$ corresponds to $\TC$, the 2-core of the giant component of $\cG(n,p)$, and the graph $G$ corresponds to its kernel $\K$.

We can clearly couple $G$ and $H$ such that each 2-path $P$ in $H$, corresponding to some edge $e \in E(G)$, would satisfy $\left||P|-w(e)\right| \leq 1$.

By Theorem~\ref{thm-decay}, the diameter of the weighted graph $G$ is w.h.p.
\[ \diam(G) = \big(\tfrac5{3}+o(1)\big)(1/\lambda)\log N\,,\]
and by Lemma~\ref{lem-dist-path-len} the path achieving it consists of $O(\log N)$ edges. Hence, recalling that $1/\lambda = (1+o(1))(1/\epsilon) \to \infty$, the distance between any two vertices $u,v\in K$ in the graph $H$ differs from their distance in $G$ by at most $O(\log N) = o(\diam(G))$, and so w.h.p.
\[ \max_{u,v\in K} \dist_{H}(u,v) = \big(\tfrac5{3}+o(1)\big)(1/\epsilon)\log (\epsilon^3 n)\,.\]
Furthermore, by the coupling of $G$ and $H$, given the metric graph $\mathcal{X}(G)$ we clearly have that distance between two given points $x,y\in\mathcal{X}$ is up to a difference of $2$ the distance between two other points $u,v\in H$ (simply take the closest points to $x,y$ in the subdivision of the corresponding edges), and vice versa. Since we know by Theorem~\ref{thm-diameter-cont} that w.h.p.\
\[ \diam(\mathcal{X}) = (2+o(1))(1/\lambda)\log N\,,\]
we can now deduce that w.h.p.
\[\diam(H) = (2+o(1))(1/\epsilon)\log (\epsilon^3 n)\,,\]
as required.
\end{proof}

\subsection{The diameter of the giant component}
We next wish to prove Theorem~\ref{mainthm-1}, which establishes the asymptotics of the diameter of the giant component.

The next lemma estimates the diameter of a Poisson-Galton-Watson tree. Throughout the proof, let $0 < \mu < 1$ be
some function of $n$ satisfying \[\mu = 1 - \epsilon + O(\epsilon^2)\,.\]
\begin{lemma}\label{lem-PGW-diameter}
Let $T$ be a $\mathrm{PGW}(\mu)$-tree for $\mu$ as above, and let $L_k$ be the $k$-th level of $T$. For any $k \geq 1/\epsilon$ we have
$ \P(L_k \neq \emptyset) = \Theta\left(\epsilon \exp\left[-k(\epsilon + O(\epsilon^2))\right]\right)$.
\end{lemma}
\begin{proof}
Let $T'$ be a Galton-Watson tree with a Binomial offspring distribution $\Bin(b, \mu/b)$.  Then $T'$ is precisely the open cluster containing the root after percolating on a $b$-ary tree $T_b$, where the percolation probability is $\mu/b$. Let
$L'_k$ be the $k$-th level of $T'$. We will use the next lemma which gives a sharp estimate for the probability that $L'_k$ is non-empty.

\begin{lemma}[\cite{Lyons}*{Theorem 2.1}, restated]\label{lem-percolation-resistance}
Assign each edge e from level $h-1$ to level $h$ in $T_b$ the edge resistance $r_e = (1-\frac{\mu}b)(\frac{\mu}b)^{-h}$. Let $R_k$ be the effective
resistance from the root to level $k$ of $T_b$. Then,
$$\frac{1}{1 + R_k}
\leq \P(L'_k \neq \emptyset)\leq \frac{2}{1 + R_k} ~.$$
\end{lemma}
In our case, the resistance $R_k$ satisfies (see, e.g., \cite{Peres}*{Example 8.3})
\begin{align*}
R_k =\sum_{i=1}^k \frac{(1-\frac{\mu}{b}) (\frac{\mu}{b})^{-i}}{b(b-1)^{i-1}} = \frac{b-1}{b-\mu}\cdot\frac{\big[(1 + \frac{1}{b-1})/\mu\big]^{k+1} -
1}{\big[(1 + \frac{1}{b-1})/\mu\big] -1 }\,.
\end{align*}
Note that for $k \geq 1/\epsilon$ we have
\[ \lim_{b\to \infty}R_k = \frac{(1/\mu)^{k+1} -1}{(1/\mu)-1} =  \Theta\left(\epsilon^{-1}\exp\left[ k(\epsilon + O(\epsilon^2))\right]\right)\,.\]
Applying Lemma \ref{lem-percolation-resistance}, we obtain that
\[\frac{1}{1 + R_k}
\leq \P(L'_k\neq \emptyset)\leq \frac{2}{1 + R_k}\,.\]
Letting $b\to \infty$ and using the fact that $\Bin(b, \frac{\mu}{b})$ converges to a $\mathrm{Poisson}(\mu)$ distribution, we obtain that for any $k\geq 1/\epsilon$,
$$ \P(L_k \neq \emptyset) = \Theta\left(\epsilon \exp\left[-k(\epsilon + O(\epsilon^2))\right]\right)\,,$$
as required.
\end{proof}

We are now ready to prove the main result.
\begin{proof}[\emph{\textbf{Proof of Theorem~\ref{mainthm-1} for the regime $\epsilon=o(n^{-1/4})$}}]
For $N$ as defined in Theorem~\ref{thm-struct}, let $G \sim \cG(N , 3)$ be a random 3-regular graph with i.i.d.\ $\Exp(\lambda)$ edge weights, denoted by $\{w(e):w\in E(G)\}$, where
$\mathrm{e}^{-\lambda} = 1-\epsilon$.
Again, let $H$ be the unweighted graph obtained by taking the underlying graph of $G$, and replacing each of its edges by paths of length i.i.d.\ geometric variables with mean $1/\epsilon$. 

In what follows, we will shift between the shortest distances in the metric graph $\mathcal{X}(G)$ and those in the weighted graph $H$. Since we have
\[ 1/\lambda = 1 / \log(1/(1 - \epsilon)) = (1/\epsilon) + O(1)\,,\] Lemma~\ref{lem-dist-path-len}, together with the aforementioned coupling between the two models, implies that this shift will only cause an error of $O(\log (\epsilon^3 n))$, which is easily absorbed in our estimate for the diameter.

Set $\delta > 0$. and consider $m = \Theta(\epsilon^2 n)$ i.i.d.\ PGW($1-\epsilon$)-trees.
By Lemma~\ref{lem-PGW-diameter}, the probability that a given such tree will have height at least
\[ h^- = (1-\delta) (1/\epsilon) \log (\epsilon^3 n) \]
is, for some $c=c(\delta)> 0$, at least
\[c \epsilon (\epsilon^3 n)^{-(1-\delta)(1+ O(\epsilon))} = c \epsilon (\epsilon^3 n)^{-1+\delta+o(1)} \deq \zeta\,.\]
Therefore, standard estimates for the binomial variable $\Bin(m,\zeta)$ (whose mean is $m \zeta = (\epsilon^3 n)^{\delta-o(1)}$)
imply that there exist at least $2$ such trees w.h.p.

Let $H_u$ denote the height of the PGW-tree attached to a vertex $u$ in the 2-core. By the above discussion, the two vertices $u,v$ with largest $H_u,H_v$ have $H_u,H_v \geq h^-$, and clearly they are uniformly distributed among the vertices of $\TC$ (by the definition of our model). As such, Lemma~\ref{lem-typical-dist-lower} asserts that their distance in $G$ is w.h.p.\ at least $2t_0 = (1+o(1))\log N$ (neglecting the distance to their nearest kernel vertices in this lower bound) and re-scaling, we obtain that their distance is at least $(1+o(1))(1/\epsilon)\log(\epsilon^3 n)$ in the 2-core.
Altogether, the distance between a level-$H_u$ leaf of the PGW-tree of $u$ and a level-$H_v$ leaf of the PGW-tree of $v$ is at least
\[ (3-2\delta+o(1))(1/\epsilon)\log(\epsilon^3 n)\,,\]
and letting $\delta\to 0$ completes the lower bound.

For the upper bound, let $B = \{u \in \TC: H_u \geq 1/\epsilon\}$. Applying Lemma~\ref{lem-PGW-diameter} and Markov's inequality (as $\E|B| = O(\epsilon^3n)$) we infer that w.h.p.\
\begin{equation}
  \label{eq-B-bound}
  |B| \leq  \epsilon^3 n \log \log (\epsilon^3 n)\,.
\end{equation} Recall that the attached trees are independent of the $2$-core and hence $B$ is independent of the structure of the $2$-core. Moreover, for any $u$,
\[\P\left(H_u \geq 1/\epsilon + \ell \given u \in B\right) = O\big(\mathrm{e}^{-(1 + O(\epsilon))\epsilon \ell}\big)\,,\]
since we have $\P(u\in B) = \Theta(\epsilon)$.

Now, let $u'$ be the closest kernel point to $u$ and let $\de_u$ be the kernel edge incident to $u$ and ending at $u'$.
Define $T_u = T_{\de_u}(q) + \frac12 w(\de_u)$, where $T_{\de_u}(q)$ is as given in~\eqref{eq-t(q)-def}, i.e.\ the minimum time $t$ at which $B_t(\de_u)$ reaches size $q =2\sqrt{N\log N}$ (adjusted by $\frac12w(\de_u)$ in the case of rare edges).

For $\delta > 0$, set
\[ t_n = \left(\tfrac12+\delta\right) (1/\epsilon) \log (\epsilon^3 n)\,.\]

Since $H_u$ and $T_u$ are independent, using \eqref{eq-tau-decay-uniform-e-cont} we deduce that for any $u$,
\begin{align*}
\P( H_u \geq \ell_1\given u \in B~,~T_u) &\leq O(\mathrm{e}^{-(1 + O(\epsilon))\epsilon\ell_1}) \,,\\
\P( T_u \geq t_n + \ell_2\given u \in B~,~H_u) &\leq
O(\mathrm{e}^{-(1 + O(\epsilon))\epsilon\ell_2})+ O(N^{-3/2})\,.
\end{align*}
We next wish to bound the upper tail of $H_u + T_u$. To this end,
let $W$ be a random variable with $\P(W \geq x) =
\mathrm{e}^{-(1+O(\epsilon))\epsilon x}$ and $Z$ be a random
variable on $\{0, n\}$ with $\P(Z = n) = O(N^{-3/2})$. Pick
independent variables $W_1$ and $W_2$ distributed as $W$ and suppose
$Z$ is independent of $W_1, W_2$.
Observe that our previous results give that
\begin{equation}
  \label{eq-Hu-Tu-dom}
  (H_u \mid u\in B~, ~T_u) \preceq W_1 \mbox{ and } (T_u - t_n \mid u\in B~, ~H_u) \preceq W_2 + Z\,.
\end{equation}
By our assumption on $W$,
\[\P(\mathrm{e}^{(1-\delta)\epsilon W} \geq x) = \P(W \geq \log x/((1-\delta)\epsilon)) \leq x^{-\frac{1+O(\epsilon)}{1-\delta}}\]
and therefore $\E\mathrm{e}^{(1-\delta)\epsilon W} \leq O(1/\delta)$
assuming $\epsilon = o(\delta)$. We now obtain that
$\E\mathrm{e}^{(1-\delta)\epsilon (W_1+W_2)} \leq O(1/\delta^2)$ and
thus $\P(W_1 + W_2 \geq x) \leq O(\mathrm{e}^{-(1-\delta)\epsilon x}/\delta^2)$.
This implies that
$$\P(W_1 + W_2 + Z \geq x ) \leq O(\mathrm{e}^{-(1-\delta)\epsilon x}/\delta^2) + O(N^{-3/2})\,.$$
and plugging this in~\eqref{eq-Hu-Tu-dom} we conclude that
\[ \P( H_u + T_u \geq t_n + \ell\given u \in B) \leq O(\mathrm{e}^{-(1 -\delta)\epsilon\ell}/\delta^2) + O( N^{-3/2})~. \]
Taking \[ \ell = (1+2\delta)(1/\epsilon) \log(\epsilon^3 n)\] and
recalling the bound \eqref{eq-B-bound} on $|B|$, we conclude that
w.h.p.\ every $u \in B$ satisfies
\[H_u+T_u \leq \big((\tfrac{3}{2}+3\delta)\epsilon^{-1}\log (\epsilon^3 n)\big)\,.\]
It remains to treat vertices not in $B$. Again applying \eqref{eq-tau-decay-uniform-e-cont}, w.h.p.\ we have
$T_u \leq (1+\delta)(1/\epsilon) \log(\epsilon^3 n)$ for all $u\in \TC$, and thus for all $u \in \TC \setminus B$,
\[H_u +T_u \leq (1/\epsilon) + (1+\delta)(1/\epsilon) \log(\epsilon^3 n)\,.\]

Altogether, we conclude that w.h.p.\ every $u\in \TC$ satisfies
\[H_u+T_u \leq \big((\tfrac{3}{2}+3\delta)(1/\epsilon) \log (\epsilon^3 n)\big)\,,\]
and Eq.~\eqref{eq-metric-dist(x,y)} now concludes the proof of the upper bound.
\end{proof}

\section{Diameters in the supercritical giant component}\label{sec:giant}

The goal of this section is to extend the proofs of Theorem~\ref{mainthm-1},\ref{mainthm-2}, provided in the previous section for the special case of $\epsilon=o(n^{-1/4})$, to any $\epsilon=o(1)$.

The proofs follow from essentially the same arguments, by replacing Theorem~\ref{thm-struct} with its following more general form:

\begin{theorem}[\cite{DKLP}*{Theorem 2}]\label{thm-struct-gen}
Let $\GC$ be the largest component of $\cG(n,p)$ for $p = \frac{1 + \epsilon}{n}$, where $\epsilon^3 n\to \infty$ and $\epsilon\to 0$. Let $\mu<1$ denote the conjugate of $1+\epsilon$, that is,
$\mu\mathrm{e}^{-\mu} = (1+\epsilon) \mathrm{e}^{-(1+\epsilon)}$. Then $\GC$ is one-sided contiguous to the following model $\tGC$:
\begin{enumerate}[\indent 1.]
  \item\label{item-struct-gen-degrees} Let $\Lambda\sim \mathcal{N}\left(1+\epsilon - \mu, \frac1{\epsilon n}\right)$ and assign
  i.i.d.\  variables $D_u \sim \mathrm{Poisson}(\Lambda)$ ($u \in [n]$) to the vertices, conditioned that $\sum D_u \one_{D_u\geq 3}$ is even.

  Let $N_k = \#\{u : D_u = k\}$ and $N= \sum_{k\geq 3}N_k$. Select a random multigraph $\K$ on $N$ vertices, uniformly among all graphs with $N_k$ vertices of degree $k$ for $k\geq 3$.
  \item\label{item-struct-gen-edges} Replace the edges of $\K$ by paths of lengths i.i.d.\ $\Geom(1-\mu)$. 
  \item\label{item-struct-gen-bushes} Attach an independent $\mathrm{Poisson}(\mu)$-Galton-Watson tree to each vertex.
\end{enumerate}
That is, $\P(\tGC \in \mathcal{A}) \to 0$ implies $\P(\GC \in \mathcal{A}) \to 0$
for any set of graphs $\mathcal{A}$.
\end{theorem}

Indeed, a Taylor-expansion of the above defined parameter $\mu$ shows that $\mu=1-\epsilon + O(\epsilon^2)$, hence our treatment of the re-scaling of the weighted graph by paths of length i.i.d.\ $\Geom(1-\mu)$ will be essentially the same as in the case of i.i.d.\ $\Geom(\epsilon)$ variables, and the same applies to the PGW($\mu$)-trees (rather than PGW($1-\epsilon$)-trees). Step~1, however, is somewhat different here, as the degree distribution of the kernel is richer. Nevertheless, with only minor modifications the original proofs will hold for this case as well. We next address these  adjustments that one needs to make.

\subsection{The 2-core and kernel}

To extend the upper bounds from the case $\epsilon=o(n^{-1/4})$ to the general case, consider the kernel $\K$. While before $\K$ was a random $3$-regular graph, now it has a degree distribution of a truncated Poisson, resulting in an order $\epsilon^k n$ expected number of vertices of degree $k$ for $k\geq 3$. As we will now show, this will only assist us in establishing the upper bound on the diameter.

Consider the weighted graph $G$ studied in Section~\ref{sec:rand-reg}.
A crucial element in the proofs of all the upper bounds was the decay of the distance of two typical vertices in the 2-core.
This was achieved by analyzing the exploration process starting from a uniformly chosen vertex. While previously every vertex had $3$ half-edges, now the degree distribution is more complicated: \begin{itemize}
  \item On one hand, encountering a vertex of large degree in the exploration process would contribute extra half-edges to our boundary, and hence accelerate the exposure speed (by increasing the rate of the exponential waiting times $\tau_i$).
  \item On the other hand, larger degrees might result in a larger tree-excess, slowing down the growth of the number of exposed vertices.
\end{itemize}
We now show that the effect of the second item is negligible. First, notice that w.h.p.\ the largest degree in the kernel is at most $\log N$ (in fact, it is $O(\frac{\log N}{\log\log N})$ w.h.p.). Assuming that this holds, \eqref{eq-tx-B-tau-r} can be replaced by
\begin{align*}
\begin{array}{l}
\P( \tx(B_{\tau_r}) \geq 1 ) \leq \P\big(\Bin\big(r\log N ,(r+2)\frac{\log N}N\big) \geq 1\big) = O\big( r^2 \frac{\log^2 N}N\big) = O\big(\frac{\log^7 N}N\big)\,,\\
\noalign{\medskip}
\P( \tx(B_{\tau_r}) \geq 2 ) \leq \P\big(\Bin\big(r\log N ,(r+2)\frac{\log N}N\big) \geq 2\big) = O\big( r^4 \frac{\log^3 N}{N^2} \big)=o\big( N^{-3/2}\Big)\,,
\end{array}
\end{align*}
which is indeed sufficient. Therefore, the probabilities of the events we condition on to control the tree-excess have the same effect. A straightforward stochastic domination argument for the first item above (where the larger degrees are in our favor) now completes the upper bounds specified in Section~\ref{sec:rand-reg} for the new degree distribution.

Establishing the lower bounds is slightly more delicate. Fix $\delta > 0$. By standard large deviation arguments, there exists some $c = c(\delta) > 1$ such that, with probability at least $1-\delta$,
\begin{equation}\label{eq-N-k-s}
N_3 = \big(\tfrac{4}{3} + o(1)\big) \epsilon^3 n\,, \qquad N_k \leq c \frac{(3\epsilon)^k}{k!} n ~\mbox{ for all }k \geq 4\,.\end{equation}
Set $z = \sqrt{N / \log N}$. Let $S_i$ be the number of unmatched half-edges in our explored set at time $\tau_i$ (where $\tau_i$ denotes the time of the $i$'th step) and let $X_i = S_{i} - S_{i-1}$. Recall that after $i\leq z$ steps we have at least $N_3-z$ vertices of degree $3$ and at most $N_k$ vertices of degree $k$ for $k\geq 4$, and that each matching samples a half-edge uniformly at random. With this in mind, the following variable $Y$ would stochastically dominate the degree of this random half-edge, then
\begin{align*}
\P(Y=3) &=\frac{3 ( N_3-z)}{3(N_3-z) + \sum_{j\geq 4}j N_j}\,,\\
\P(Y=  k) &=\frac{k N_{k}}{3(N_3-z) + \sum_{j\geq 4}j N_j}\quad (k =4,5,\ldots)\,.
\end{align*}
Since each matching to a newly exposed degree-$k$ vertex costs one half-edge and introduces $k-1$ new ones, it follows that $X_i \preceq Y_i -2$ where the $Y_i$'s are i.i.d.\ with $Y_i\sim Y$. Setting $\tilde{S}_t = \sum_{i=1}^t (Y_i-2)$ we get that $S_t \preceq S_0 + \tilde{S}_t $ for all $t$.

By our assumption \eqref{eq-N-k-s}, for every $k\geq 2$
\begin{align*}
  \P(Y_i-2 = k)&\leq \frac{c (k+2)(3\epsilon)^{k+2} n / (k+2)!}{(4-o(1))\epsilon^3 n} \leq c'\frac{(3\epsilon)^{k-1}}{(k+1)!}
\end{align*}
for some constant $c' > 1$.
Hence, the Laplace transform for $Y_i-2$ satisfies
\begin{align*}
\E \mathrm{e}^{\lambda (Y_i-2)} \leq \mathrm{e}^{\lambda} + \sum_{k = 2}^\infty \mathrm{e}^{\lambda k} c' \frac{(3\epsilon)^{k-1}}{(k+1)!} \leq
c' \mathrm{e}^{\lambda} \exp(\mathrm{e}^{\lambda}3\epsilon)\,.
\end{align*}
Setting $\lambda = \log(1/\epsilon)$, we
arrive at $\E \mathrm{e}^{\lambda X_i } \leq
c'\mathrm{e}^{3}/\epsilon$. Now, an application of Markov's
inequality gives that for  large $n$,
\begin{align*}
\P\left(\tilde{S}_t \geq  \Big(1 + \frac{1}{\log \log (1/\epsilon)}\Big)t\right) &\leq \Big(\frac{c'\mathrm{e}^{3}}{\epsilon}\Big)^{t} \exp\left( - \lambda t \Big(1 + \frac{1}{\log \log (1/\epsilon)} \Big)\right) \\
&\leq \left( c' \mathrm{e}^{3}\exp\Big(-\frac{\log(1/\epsilon)}{\log\log(1/\epsilon)}\Big)\right)^t  \leq 2^{-t}\,,
\end{align*}
where the last inequality holds for any large $n$. We can now infer that
\[\sum_{t\geq \sqrt{1/\epsilon}} \P\left(\tilde{S}_t \geq  \Big(1 + \frac{1}{\log \log (1/\epsilon)} \Big)t\right) = O(2^{-1/\sqrt{\epsilon}}) = o(1)\,.\]
On the other hand, as $\P(X_i \geq 2) \leq \P(Y_i \geq 4) = O(\epsilon)$, it follows that
\[ \P\left( X_i = 1 \mbox{ for all $1\leq i \leq 1/\sqrt{\epsilon}$}\right) \geq 1 - O(\sqrt{\epsilon}) = 1-o(1)\,. \]
Altogether, we conclude that given \eqref{eq-N-k-s}, we have that w.h.p.\
\begin{equation}\label{eq-bound-S-t} S_t \leq S_0 + t +
\frac{t}{\log\log(1/\epsilon)} = (1+ o(1)) t + 3\quad\mbox{ for all $k=1,2,\ldots,z$}\,,\end{equation}
where we used the fact that $S_0$ is the degree of the starting vertex for the
exploration and is thus equal to 3 w.h.p.

Letting $\delta \to 0$ (recall its definition above~\eqref{eq-N-k-s}), we obtain that \eqref{eq-bound-S-t} holds w.h.p.
Therefore, we may perform the exploration process according to the argument of Lemma~\ref{lem-typical-dist-lower}, and at every step $t$ condition that the event in \eqref{eq-bound-S-t} indeed holds up to that point, and so the estimate on $\tau_z$ immediately follows in the same manner as before. Finally, recall that the lower bound for the distance between kernel points was obtained from the above ingredients together with the asymptotics of the number of kernel vertices via a simple second moment argument. These asymptotics are the same in our new setting, thus completing the proof of the lower bound for the kernel. An analogous argument gives the corresponding lower bound for the metric graph and the 2-core (here the asymptotic number of edges plays a part in the second moment argument).

Finally, note that, as $\mu = 1 - \epsilon + O(\epsilon^2)$, writing the final estimate involving the term $1/\epsilon$ rather than $1/(1-\mu)$ results in a multiplicative factor of $1+O(\epsilon) = 1 + o(1)$, keeping the statement valid.

\subsection{The giant component} Having extended the treatment of the 2-core and kernel to the case of $\epsilon=o(1)$, it remains to address
 the attached trees. Note that Lemma~\ref{lem-PGW-diameter} applies directly for the value of $\mu$ as defined in Theorem~\ref{thm-struct-gen},
 hence our estimates for $H_u$, the height of the tree attached to a vertex $u$ of the 2-core, remain unchanged. As the rest of the arguments
  are applications of the results for weighted graphs and metric graphs (already discussed in the previous subsection), they hold without modification.

\section*{Acknowledgments}
We wish to thank Asaf Nachmias for helpful discussions at an early stage of this project and an anonymous referee for thorough comments.

\end{document}